\documentclass{elsarticle}

\usepackage{lineno,hyperref}
\modulolinenumbers[1]

\usepackage{amssymb}

\journal{Finite Fields and Their Applications}

\newcommand{\Tr}{\textrm{\rm Tr}}
\newcommand{\N}{\textrm{\rm N}}
\newcommand{\ord}{\textrm{\rm ord}}
\newcommand{\qord}{\textrm{\rm qord}}

\newcommand{\residuo}[2]{(#1 \mbox{ mod } #2)}

\begin{document}

\def\bbbr{{\rm I\!R}} 
\def\bbbm{{\rm I\!M}}
\def\bbbn{{\rm I\!N}} 
\def\bbbf{{\rm I\!F}}
\def\bbbh{{\rm I\!H}}
\def\bbbk{{\rm I\!K}}
\def\bbbp{{\rm I\!P}}
\def\bbbz{{\mathchoice {\hbox{$\sf\textstyle Z\kern-0.4em Z$}}
{\hbox{$\sf\textstyle Z\kern-0.4em Z$}}
{\hbox{$\sf\scriptstyle Z\kern-0.3em Z$}}
{\hbox{$\sf\scriptscriptstyle Z\kern-0.2em Z$}}}}

\newtheorem{definition}{Definition}
\newtheorem{proposition}{Proposition}
\newtheorem{theorem}{Theorem}
\newtheorem{lemma}{Lemma}
\newtheorem{corollary}{Corollary}
\newtheorem{remark}{Remark}
\newtheorem{example}{Example}
\newtheorem{acknowledgments}{Acknowledgments}

\newenvironment{proof}{\begin{trivlist}\item[]{\em Proof: }}%
{\samepage \hfill{\hbox{\rlap{$\sqcap$}$\sqcup$}}\end{trivlist}}

\begin{frontmatter}

\title{Necessary and sufficient conditions on the order of a finite field $\bbbf_q$ for the easy identification of primitive polynomials of degree 2\tnoteref{mytitlenote}}
\tnotetext[mytitlenote]{Manuscript partially supported by PAPIIT-UNAM IN116626.}


\author{Gerardo Vega}

\address{Direcci\'on General de C\'omputo y de Tecnolog\'{\i}as de Informaci\'on y Comunicaci\'on, Uni\-ver\-si\-dad Nacional Aut\'onoma de M\'exico, 04510 Ciudad de M\'exico, MEXICO \\
(e-mail: gerardov@unam.mx).}

\begin{abstract} 
We present the necessary and sufficient conditions on the order $q$ of a finite field $\bbbf_q$ such that every irreducible polynomial of the form $x^2+bx+c \in \bbbf_q[x]$, with $b\neq 0$ and $c$ a primitive element of $\bbbf_{q}$, is a primitive polynomial. As a by-product of this result, we also present a new infinite family of finite fields $\bbbf_q$ for which it is easy, in a different way, to determine when an irreducible polynomial of degree two is primitive.
     
\end{abstract}


\end{frontmatter}


\section{Introduction}\label{secuno}

Let $q$ be a power of a prime number and $\bbbf_q$ a finite field with $q$ elements. In Finite Field Theory, a primitive polynomial is an irreducible polynomial of degree $m$ over $\bbbf_q$ that has the property that any of its roots is a primitive element of the finite field extension $\bbbf_{q^m}$. Hence, from a mathematical point of view, primitive polynomials and their roots have been of great interest in research areas such as Number Theory, Combinatorics, and Algebraic Geometry. Primitive polynomials are also widely used in various applications of finite fields such as Coding Theory and Cryptography. For example, primitive polynomials can be used to create pseudo-random number generators of huge periods, which is of great importance for cryptographic applications. For a better understanding of primitive polynomials, the reference list in \cite{Vega} is useful.

It is quite easy to determine when a monic polynomial of degree two with coefficients in the field of real numbers is irreducible over that field. More specifically, if $f(x)=x^2+bx+c \in \bbbr[x]$, then it is well-known that:

$$f(x)\; \mbox{ is irreducible over $\bbbr$}\; \Longleftrightarrow \;b^2-4c<0\;,\;\mbox{ that is }\;\sqrt{\;b^2-4c\;}\not\in \bbbr\;.$$ 

\noindent
Now suppose that $f(x)=x^2+bx+c \in \bbbf_q[x]$ is irreducible. Thus, in this case, the following question arises naturally: Is there a similar test --again in terms of the two coefficients $b$ and $c$-- to determine whether $f(x)$ is a primitive polynomial in $\bbbf_q[x]$? That is

$$f(x)\; \mbox{ is primitive over $\bbbf_q$}\; \Longleftrightarrow \;b\;\mbox{ and }\; c\; \mbox{ satisfy some test to be determined?}$$

\noindent
Despite the simplicity of this question, there is no general answer. However, some infinite families of finite fields $\bbbf_q$ for which it is easy to determine when an irreducible polynomial of degree two is primitive, were recently identified in \cite{Vega}. More specifically, sufficient conditions were identified on the order $q$ of the finite field $\bbbf_q$ so that it is easy to determine in that field when an irreducible polynomial of degree two is primitive:

\begin{theorem}\label{teoceroA}
\cite[Theorem 4]{Vega} Let $q$ be a power of a prime number. If $q+1$ is either of the form $q+1=2^t$, $q+1=\pi$, or $q+1=2 \pi$, for some positive integer $t$ and some prime $\pi>2$, then any irreducible polynomial of the form $f(x)=x^2+bx+c \in \bbbf_{q}[x]$, with $b\neq 0$ and $c$ a primitive element of $\bbbf_q$, is a primitive polynomial.
\end{theorem}

Basically, what the previous theorem tells us is that there exist at least three infinite families of finite fields for which, given an irreducible polynomial of the form $f(x)=x^2+bx+c$, it is very easy to determine whether it is primitive simply by looking at its two coefficients $b$ and $c$. For example, let $q=13$ and note that $q+1=2(7)$. Note also that the only primitive elements of the finite field $\bbbf_{13}$ are the four elements $2$, $6$, $7$, and $11$. Furthermore, it is not difficult to verify that the three polynomials $x^2+2$, $x^2+5x+1$, and $x^2+2x+7$ are irreducible over $\bbbf_{13}$; however, by Theorem \ref{teoceroA}, only the last one is primitive.

The first finite fields $\bbbf_{q}$ that satisfy the sufficient conditions in Theorem \ref{teoceroA} are shown in \cite[Table I]{Vega}. As an erratum, note that the finite field $\bbbf_{61}$ was, by mistake, excluded from that table ($q+1=62=2(31)$).

Now, an interesting question is: Are the sufficient conditions in Theorem \ref{teoceroA} also necessary? In other words, does there exist another finite field $\bbbf_q$ such that, for every irreducible polynomial of the form $f(x)=x^2+bx+c$, it is primitive if and only if $b\neq 0$ and $c$ is a primitive element of the finite field? The main objective of this work is to demonstrate that the sufficient conditions of Theorem \ref{teoceroA} are indeed also necessary. That is, our main objective is to formally show that, beyond the three infinite families of finite fields from Theorem \ref{teoceroA}, there is no other finite field $\bbbf_q$ such that, given any irreducible polynomial of the form $f(x)=x^2+bx+c$, it will be primitive if and only if $b\neq 0$ and $c$ a primitive element of that finite field. In mathematics, and in many other areas of science, it is generally more difficult to prove that something does not exist than to prove that it does exist. As a byproduct of the proof of this result, we present another infinite family of finite fields $\bbbf_q$ for which, using a different test involving the two coefficients $b$ and $c$, it is easy to determine when an irreducible polynomial of degree two is primitive.

This paper is structured as follows: In Section \ref{secdos}, we establish notation and recall key definitions. To ensure the paper is relatively self-contained, we also recall some results from \cite{Vega}. Section \ref{sectres} is devoted to presenting some preliminary results. In particular, we present an interesting connection between the order and the quasi-order (see definitions below) of an irreducible polynomial of any positive degree. In Section \ref{seccuatro}, fixing a primitive element $\alpha$ of $\bbbf_q$, we obtain the exact number of irreducible and primitive polynomials of the form $x^2+\alpha^i x+\alpha$, with $0\leq i<q-1$. This result is used in Section \ref{seccinco} to prove that the sufficient conditions of Theorem \ref{teoceroA} are also necessary, while in Section \ref{secseis} we present a new infinite family of finite fields in which it is also easy to determine when an irreducible polynomial is primitive. Finally, Section \ref{secsiete} will be devoted to the conclusions.

\section{Notation, definitions, and already-known results}\label{secdos}

First of all, we establish, for this section and for the rest of this paper, the following:

\medskip
\noindent
{\bf Notation.} For positive integers $i$, $m$ and $z$ with $m>1$ and $0\leq z<m$, when we write $z=\residuo{i}{m}$, rather than $z \equiv i \!\pmod{m}$, we mean that $z$ equals the remainder that results when $i$ is divided by $m$. By $q$ we denote a power of a prime number. We are going to fix $\delta=0$ if $q$ is even, and $\delta=1$ otherwise. From now on, $\gamma$ will denote a primitive element of $\bbbf_{q^2}$ and we fix $\alpha=\gamma^{q+1}$ as a primitive element of $\bbbf_{q}$. For any integer $0 \leq i < q-1$, we define ${\cal C}_i^{(q-1,q^2)}:=\gamma^i \langle \gamma^{q-1} \rangle$, where $\langle \gamma^{q-1} \rangle$ denotes the subgroup of $\bbbf_{q^2}^*$ generated by $\gamma^{q-1}$. The $q-1$ cosets, ${\cal C}_i^{(q-1,q^2)}$, are called the {\em cyclotomic classes} of order $q-1$ in $\bbbf_{q^2}$. 

By using ``$\Tr_{\bbbf_{q^2}/\bbbf_q}$'' and ``$\N_{\bbbf_{q^2}/\bbbf_q}$'' we will denote, respectively, the {\em trace mapping} and the {\em norm mapping} from $\bbbf_{q^2}$ to $\bbbf_q$. Let $n>0$ be an integer and let $p(x)$ be a monic polynomial in $\bbbf_q[x]$; then by $NI_q(p(x),n)$ we will denote the number of monic irreducible polynomials of degree $n$ in $\bbbf_q[x]$ that divide $p(x)$.

The following three definitions are key to this work:

\begin{definition}\label{defPP}
\cite[Definition 3.15]{Lidl} A polynomial $f(x)\in \bbbf_{q}[x]$ of degree $m\geq 1$ is called a {\em primitive polynomial} over $\bbbf_{q}$ if it is the minimal polynomial over $\bbbf_{q}$ of a primitive element of $\bbbf_{q^m}$.
\end{definition}

\begin{definition}\label{defcero}
\cite[Theorem 3.2]{Lidl} Let $f(x) \in \bbbf_q [x]$ be a polynomial of positive degree with $f(0) \neq 0$. The least positive integer $e$ for which $f(x)$ divides $x^e-1$ is called the {\em order} of $f(x)$ and denoted by \ord$(f(x))$. 
\end{definition}

\begin{definition}\label{defuno}
Let $f(x) \in \bbbf_q [x]$ be a polynomial of positive degree with $f(0) \neq 0$. The least positive integer $\rho$ for which $x^{\rho}$ is congruent modulo $f(x)$, to some element $a \in \bbbf_q^*$, is called the {\em quasi-order} of $f(x)$ and denoted by \qord$(f(x))$. That is, $x^{\rho} \equiv a \pmod{f(x)}$ and $\rho=\mbox{\qord}(f(x))$. 
\end{definition}

Due to its importance for this work, we recall the following:  

\begin{remark}\label{remuno} 
\cite[Remark 1]{Vega} Let $f(x)$, $\rho$, and $a$ be as in Definition \ref{defuno} and let $e,y,r$ be positive integers such that $e=\rho y + r$, where $0\leq r < \rho$. If $r\neq 0$, then $x^r \equiv g(x) \pmod{f(x)}$ for some non-constant polynomial $g(x) \in \bbbf_q [x]$, and $x^{e}=x^{\rho y}x^r \equiv x^{\rho y}g(x) \equiv a^y g(x) \pmod{f(x)}$. Therefore, for a suitable positive integer $e$ and a suitable element $d\in \bbbf_q^*$, $x^e \equiv d \pmod{f(x)}$ iff $\rho | e$ and $d=a^{e/\rho}$.
\end{remark}

The following is a new characterization of all primitive polynomials of degree two:

\begin{theorem}\label{MiCar2}
\cite[Theorem 5]{Vega} Let $f(x)=x^2+bx+c \in \bbbf_{q}[x]$ be irreducible. In the case that none of the three sufficient conditions in Theorem \ref{teoceroA} are satisfied, let $p$ be the smallest odd prime that divides $q+1$ and let $h(x), r(x) \in \bbbf_q [x]$ be the uniquely determined polynomials such that $h(x)=\frac{x^{\frac{q+1}{p}+1}-r(x)}{f(x)}$, where $\deg(r(x))<2$. Then $f(x)$ is primitive iff $b\neq 0$, $c$ is a primitive element of $\bbbf_{q}$ and at least one of the following two conditions holds:

\begin{enumerate}
\item[{\rm (A)}] $q+1$ is either of the form $q+1=\pi$, $q+1=2^t$, or $q+1=2 \pi$, for some positive integer $t$ and some prime $\pi>2$.
\item[{\rm (B)}] All the $\frac{q+1}{p}$ terms of $h(x)$ are non-zero.
\end{enumerate}
\end{theorem}

Let $f(x)=x^2+bx+c \in \bbbf_{q}[x]$ be irreducible, with $b\neq 0$ and $c$ a primitive element of $\bbbf_{q}$. In contrast to Examples 6 and 7 of \cite{Vega}, the following example shows that if $f(x)$ is a non-primitive polynomial (that is, Conditions {\rm (A)} and {\rm (B)} of Theorem \ref{MiCar2} are not satisfied), then the constant term of $h(x)$ is not necessarily zero.

\begin{example}\label{ejcero}
Let $q=29$. Note that $\langle 2 \rangle=\bbbf_{29}^*$, $q+1=30=2(3)(5)$, $p=3$, $\frac{q+1}{p}=10$, and therefore the finite field $\bbbf_{29}$ does not satisfy Condition {\rm (A)}. Consider the polynomials $x^2+5x+2$, $x^2+4x+2$, and $x^2+8x+2$. It is not difficult to see that these three polynomials are irreducible over $\bbbf_{29}$ and 

\begin{eqnarray}
\frac{x^{11}-(27x+3)}{(x^2+5x+2)} &=& x^9+24x^8+23x^7+11x^6+15x^5+ \nonumber \\
&& 19x^4+20x^3+7x^2+12x+13\;, \nonumber \\
\frac{x^{11}-26x}{(x^2+4x+2)} &=& x^9+25x^8+14x^7+10x^6+19x^5+ \nonumber \\
&& 20x^4+27x^3+26x^2+16x+0\;,\;\mbox{ and}  \nonumber \\
\frac{x^{11}-(26x+5)}{(x^2+8x+2)} &=& x^9+21x^8+4x^7+13x^6+4x^5+ \nonumber \\
&& 0x^4+21x^3+6x^2+26x+12 \;. \nonumber
\end{eqnarray}

\noindent
Thus, note that $x^2+5x+2$ satisfies Condition {\rm (B)} and therefore it is a primitive polynomial. However the polynomials $x^2+4x+2$ and $x^2+8x+2$ do not satisfy Condition (B) in different ways, and therefore they are non-primitive. In fact, it is not difficult to see that \ord$(x^2+5x+2)=840$, \ord$(x^2+4x+2)=280$, and \ord$(x^2+8x+2)=168$.
\end{example}

\begin{remark}\label{remAA}
Let $f(x)=x^2+bx+c \in \bbbf_{q}[x]$ be any irreducible polynomial, with $b\neq 0$ and $c$ a primitive element of $\bbbf_{q}$. As we will see later in Theorem \ref{MiCar2E}, by using the polynomial $h(x)$ in Theorem \ref{MiCar2}, it is always possible to determine the quasi-order and the order of $f(x)$.
\end{remark}

Finally, we recall the following result that will be important for this work. For the sake of completeness, we also provide its proof.

\begin{proposition}\label{prodosA}
\cite[Proposition 2]{Vega} Let $f(x) \in \bbbf_q[x]$ be a monic polynomial of degree two, where $f(0) \neq 0$, and assume for some $a\in \bbbf_q^*$ that $x^{\rho} \equiv a \pmod{f(x)}$, with $\rho=\mbox{\qord}(f(x))$. Let 

$$g(x)=\frac{x^{\rho+1}-ax}{f(x)}=g_{1}x^{\rho-1}+g_{2}x^{\rho-2}+\ldots+g_{\rho-1}x+g_{\rho}\;.$$ 

\noindent
Then, for $1\leq i\leq \rho-1$, $g_i\neq 0$ and $g_{\rho}=0$  (i.e. apart from the constant term, all terms of $g(x)$ are non-zero). 
\end{proposition}

\begin{proof}
Assume $f(x)=x^2+bx+c$, with $c\neq0$ (observe that $b$ could be zero). Since $f(x)$ is monic and $g(x)=x\frac{x^{\rho}-a}{f(x)}$, $g_{1}=1$ and $g_{\rho}=0$. Recall that $\rho$ is the smallest positive integer such that $f(x)|(x^{\rho}-a)$, for some element $a \in \bbbf_q^*$. Let $g_2=-b$ and note that $x^2=g_1f(x)+g_2x-g_1c$. If $g_2=0$ then $\rho=2$ and $a=-g_1c \in \bbbf_q^*$. On the contrary, if $g_2\neq 0$ then let $g_3=-(g_2b+g_1c)$ and note that $x^3=(g_1x+g_2)f(x)+g_3x-g_2c$. If $g_3=0$ then $\rho=3$ and $a=-g_2c \in \bbbf_q^*$. On the contrary, if $g_3\neq 0$ then let $g_4=-(g_3b+g_2c)$ and note that $x^4=(g_1x^2+g_2x+g_3)f(x)+g_4x-g_3c$. Clearly we can continue in the same way and get eventually 

$$x^{\rho}=(g_{1}x^{\rho-2}+g_{2}x^{\rho-3}+\ldots+g_{\rho-1})f(x)+g_{\rho}x-g_{\rho-1}c\;,$$ 

\noindent
where $g_{\rho}=0$ and $a=-g_{\rho-1}c \in \bbbf_q^*$.
\end{proof}

\section{Preliminary results}\label{sectres}

As an initial preliminary result, we present the following theorem that provides an interesting connection between the order and the quasi-order of any irreducible polynomial of positive degree.
 
\begin{theorem}\label{teocero}
Let $h(x) \in \bbbf_q [x]$ be a polynomial of positive degree with $h(0) \neq 0$. Let $r=\ord(h(x))$ and $\rho=\qord(h(x))$. If $h(x)$ is an irreducible polynomial, then $r=\gcd(r,q-1)\rho$.
\end{theorem}

\begin{proof}
Let $F$ be the splitting field of $h(x)$. By \cite[Lemma 3.17]{Lidl}, there must exist an integer $d$ such that $d | (q-1)$ and $r=d \rho$. Then we have $d | \gcd(r,q-1)$. Suppose for the sake of contradiction that $d < \gcd(r,q-1)$. Thus we have $r = \gcd(r,q-1) \rho / t$, for some integer $t > 1$, which implies that $\rho > \rho / t = r/\gcd(r,q-1)$. Let $\beta$ be a fixed root of $h(x)$ in $F$. Since $\beta$ has order $r$ (see \cite[Theorem 3.33]{Lidl}), $\beta^{r(q-1)/\gcd(r,q-1)}=1$; hence, $\beta^{r/\gcd(r,q-1)} \in \bbbf_q$. Setting $\beta^{r/\gcd(r,q-1)}=a$. This means that $\beta$ is a root of $x^{\rho/t}-a$, and consequently $h(x) | (x^{\rho/t}-a)$. This yields a contradiction since $\rho / t < \rho$.
\end{proof}

Observe that the converse of the previous theorem is not true. For example, if $q=3$ and $h(x)=(x+1)^2$, then $r=6$, $\rho=3$ and $\gcd(r,q-1)=2$. 

\medskip
The following is an important consequence of our previous theorem.

\begin{proposition}\label{procero}
Let $f(x) \in \bbbf_q [x]$ be a polynomial of positive degree with $f(0) \neq 0$ and let $a \in \bbbf_q^*$ and $\rho, r \in \bbbz^+$ be the uniquely determined values such that $x^{\rho} \equiv a \pmod{f(x)}$, $\rho=\qord(f(x))$, and $r=\ord(f(x))$. If $\gcd(r,q-1)<q-1$, then $a$ is not a primitive element in $\bbbf_q$.  
\end{proposition}

\begin{proof}
By Definition \ref{defcero} and Remark \ref{remuno}, $x^r \equiv 1 \pmod{f(x)}$ and $1=a^{r/\rho}$. On the other hand, by Theorem \ref{teocero}, we have $r=\gcd(r,q-1)\rho$ and therefore $1=a^{\gcd(r,q-1)}$. Since $\gcd(r,q-1)<q-1$, clearly $a$ could not be a primitive element in $\bbbf_q$.  
\end{proof}

Let $f(x)=x^2+bx+c \in \bbbf_{q}[x]$ be irreducible, with $b\neq 0$ and $c$ is a primitive element. Suppose that Condition {\rm (A)} is not satisfied (see Example \ref{ejcero} and Remark \ref{remAA}). Then, under Condition {\rm (B)}, $f(x)$ clearly may or may not be primitive. By means of the following result, which can be considered as an extension of Theorem \ref{MiCar2}, we can determine the order and the quasi-order of $f(x)$:

\begin{theorem}\label{MiCar2E}
Let $f(x)=x^2+bx+c \in \bbbf_{q}[x]$ be irreducible, with $b\neq 0$ and $c$ a primitive element of $\bbbf_q$. Suppose that $q$ does not satisfy Condition {\rm (A)} in Theorem \ref{MiCar2}. In this case, let $p$ be the smallest odd prime dividing $q+1$ and let $h(x), r(x) \in \bbbf_q [x]$ be the uniquely polynomials such that $\deg(r(x))<2$ and 

$$h(x)=\frac{x^{\frac{q+1}{p}+1}-r(x)}{f(x)}=h_{1}x^{\frac{q+1}{p}-1}+h_{2}x^{\frac{q+1}{p}-2}+\ldots+h_{\frac{q+1}{p}-1}x+h_{\frac{q+1}{p}}\;.$$ 

\noindent
If Condition {\rm (B)}, in Theorem \ref{MiCar2}, is satisfied, then \qord$(f(x))=q+1$ and \ord$(f(x))=q^2-1$. If, on the contrary, such condition is not satisfied, then let $1\leq k \leq \frac{q+1}{p}$ be the smallest integer such that $h_k=0$. Then \qord$(f(x))=k$ and \ord$(f(x))=k(q-1)$.
\end{theorem}

\begin{proof}
If Condition {\rm (B)} is satisfied, then, by Theorem \ref{MiCar2}, $f(x)$ is primitive and \qord$(f(x))=q+1$ (see \cite[Theorem 3.18]{Lidl}) and \ord$(f(x))=q^2-1$. 

In the contrary case, let $h'(x), r'(x) \in \bbbf_q [x]$ be the uniquely polynomials such that $h'(x)=\frac{x^{\frac{q+1}{p}+2}-r'(x)}{f(x)}$ and $\deg(r'(x))<2$. Thus note that 

\begin{eqnarray}
h'(x)&=&\frac{x^{\frac{q+1}{p}+2}-r'(x)}{f(x)}\;, \nonumber \\
&=&h'_{1}x^{\frac{q+1}{p}}+h'_{2}x^{\frac{q+1}{p}-1}+\ldots+h'_{\frac{q+1}{p}}x+h'_{\frac{q+1}{p}+1}\;, \nonumber \\
&=&h_{1}x^{\frac{q+1}{p}}+h_{2}x^{\frac{q+1}{p}-1}+\ldots+h_{\frac{q+1}{p}}x+h'_{\frac{q+1}{p}+1}\;, \nonumber \\
&=&xh(x)+h'_{\frac{q+1}{p}+1}\;. \nonumber
\end{eqnarray}

\noindent
Hence

$$x^k=f(x)(h_{1}x^{k-2}+h_{2}x^{k-3}+\ldots+h_{k-2}x+h_{k-1})+h_kx+{h'}_{k+1}\;.$$

\noindent
Since $h_k={h'}_{k}=0$, we have

$$\frac{x^{k}-{h'}_{k+1}}{f(x)}=h_{1}x^{k-2}+h_{2}x^{k-3}+\ldots+h_{k-2}x+h_{k-1}\;,$$

\noindent
and this means that $x^{k} \equiv {h'}_{k+1} \pmod{f(x)}$. But $f(x) \nmid x^{k}$, hence ${h'}_{k+1} \neq 0$. Therefore $k=\qord(f(x))$. Let $r=\ord(f(x))$. By \cite[Proposition 1]{Vega}, we have $f(x) | (x^{q+1}-c)$; thus, according to Remark \ref{remuno}, $k|(q+1)$ and $c=({h'}_{k+1})^{(q+1)/k}$. On the other hand, by Theorem \ref{teocero}, $r=k \gcd(r,q-1)$. Suppose toward a contradiction that \ord$(f(x))=k \gcd(r,q-1) \neq k (q-1)$. Thus $\gcd(r,q-1)<q-1$ and, by Proposition \ref{procero}, ${h'}_{k+1}$ is not a primitive element in $\bbbf_q$. But this is a contradiction because $c$ is a primitive element of $\bbbf_q$ and $c=({h'}_{k+1})^{(q+1)/k}$. 
\end{proof}

The following result, which is quite similar to \cite[Proposition 1]{Vega}, gives an explicit description of all possible irreducible divisors of the polynomial $x^{q+1}-c$, where $c \in \bbbf_q^*$. In particular, this result will be useful to us when $c=\alpha$. For completeness, we include its proof.

\begin{proposition}\label{prouno}
Let $c \in \bbbf_q^*$ and let $f(x)\neq 1$ be a monic irreducible polynomial in $\bbbf_q[x]$. Let $0\leq i < q-1$ be the unique integer such that $c=\alpha^i$. Then 

\begin{equation}\label{eqcero}
x^{q+1}-c=\prod_{r \in {\cal C}_i^{(q-1,q^2)}}(x-r)\;,
\end{equation}

\noindent
and if $f(x)$ is a divisor of $x^{q+1}-c$ then $f(x)$ is either of the form $f(x)=x^2 \pm bx+c$ or $f(x)=x \pm e$, for some elements $b,e \in \bbbf_q$. In particular, if $c=\alpha$ and $q$ is even, then $f(x)$ is either of the form $f(x)=x^2+bx+\alpha$ or $f(x)=x+\alpha^{\frac{q}{2}}$, for some element $b \in \bbbf_q^*$; and if $q$ is odd, then $f(x)$ is of the form $f(x)=x^2+bx+\alpha$, for some element $b \in \bbbf_q$ ($b$ could be zero). Moreover, $x^2+\alpha$ is an irreducible divisor of $x^{q+1}-\alpha$ iff $q$ is odd and $q \equiv 1 \!\!\pmod{4}$. 
\end{proposition}

\begin{proof}
Let $r \in {\cal C}_i^{(q-1,q^2)}$; then $r=\gamma^{i+j(q-1)}$ for some $j$. Note that $\N_{\bbbf_{q^2}/\bbbf_q}(r)=r^{q+1}=\gamma^{(q+1)i+(q^2-1)j}=\gamma^{(q+1)i}=\alpha^i=c$. Consequently, (\ref{eqcero}) holds, because $|{\cal C}_i^{(q-1,q^2)}|=q+1$ and $r^{q+1}-c=0$ iff $r \in {\cal C}_i^{(q-1,q^2)}$. Clearly $r^{q^2}-r=0$, for all $r \in \bbbf_{q^2}$. Thus, by \cite[Lemma 2.12]{Lidl}, we certainly have $(x^{q+1}-c)|(x^{q^2}-x)$ and if $f(x) | (x^{q+1}-c)$, then $f(x) | (x^{q^2}-x)$ and, by \cite[Lemma 2.13]{Lidl}, $1\leq \deg(f(x))\leq 2$. That is, $f(x)$ is a monic irreducible polynomial of degree one or two. Now, if $q$ is even then $-r=r$ and if $q$ is odd then $-r=\gamma^{\frac{q+1}{2}(q-1)}\gamma^{i+j(q-1)}=\gamma^{i+(j+\frac{q+1}{2})(q-1)}\in {\cal C}_i^{(q-1,q^2)}$. Whether or not $q$ is even, we have 

$$r^q=\gamma^{i+i(q-1)+j(q-1)(q+1-1)}=\gamma^{i+i(q-1)-j(q-1)}=\gamma^{i+(i-j)(q-1)} \in {\cal C}_i^{(q-1,q^2)}\;.$$ 

\noindent
Therefore $r\in {\cal C}_i^{(q-1,q^2)}$ iff $-r,r^q \in {\cal C}_i^{(q-1,q^2)}$. Clearly, we have $(x+r)(x+r^q)=x^2+\Tr_{\bbbf_{q^2}/\bbbf_q}(r)x+\N_{\bbbf_{q^2}/\bbbf_q}(r) \in \bbbf_q[x]$ and $(x-r)(x-r^q)=x^2-\Tr_{\bbbf_{q^2}/\bbbf_q}(r)x+\N_{\bbbf_{q^2}/\bbbf_q}(r) \in \bbbf_q[x]$, and both polynomials are irreducible over $\bbbf_q$ iff $r \in {\cal C}_i^{(q-1,q^2)} \setminus \bbbf_q$. Since $r \in {\cal C}_i^{(q-1,q^2)}$ iff $\N_{\bbbf_{q^2}/\bbbf_q}(r)=r^{q+1}=\gamma^{(q+1)i}=c$, any irreducible divisor of degree two that divides $x^{q+1}-c$ is of the form $f(x)=x^2 \pm bx+c$ for some element $b \in \bbbf_q$. If $e \in {\cal C}_i^{(q-1,q^2)} \cap \bbbf_q^*$ then $0=e^{q+1}-c=e^2-c$ and the two polynomials $f(x)=x \pm e$ are linear divisors of $x^{q+1}-c$. 

Now assume that $c=\alpha^1$. Thus, observe that if $q$ is even then $|{\cal C}_1^{(q-1,q^2)} \cap \bbbf_q^*|=1$. In fact if $q$ is even, then $\alpha^{\frac{q}{2}}=\gamma^{1}\gamma^{(q-1)(\frac{q}{2}+1)}\in {\cal C}_{1}^{(q-1,q^2)} \cap \bbbf_q^*$ and any element $d \in \bbbf_q^*$ has a unique square root and therefore any polynomial of the form $x^2+d$ is always reducible. Consequently, if $q$ is even, then $f(x)$ is either of the form $f(x)=x^2+bx+\alpha$ or $f(x)=x+\alpha^{\frac{q}{2}}$, for some element $b \in \bbbf_q^*$. Alternatively, if $q$ is odd then $|{\cal C}_1^{(q-1,q^2)} \cap \bbbf_q^*|=0$. Hence, if $q$ is odd, then $f(x)$ is always of the form $f(x)=x^2+bx+\alpha$, for some element $b \in \bbbf_q$. Finally, if $q$ is odd, an element $c \in \bbbf_q^*$ is a square iff $c=\alpha^{2j}$, for some $0\leq j < \frac{q-1}{2}$. Thus the polynomial $x^2+\alpha$ is an irreducible divisor of $x^{q+1}-\alpha$ iff $-\alpha$ is not a square iff $-1$ is a square iff $q \equiv 1 \!\!\pmod{4}$. 
\end{proof}

In light of the previous theorem, we can now determine the exact number of distinct irreducible polynomials of degree two that divide $x^{q+1}-\alpha$.

\begin{corollary}\label{coruno}
Assume our current notation. Then

$$NI_q(x^{q+1}-\alpha,2)=\left\{ \begin{array}{cl}
		\;\frac{q}{2} & \mbox{ if $q$ is even}, \\
\\
		\;\frac{q+1}{2} & \mbox{ if $q$ is odd.}
			\end{array}
\right .$$
\end{corollary}

\begin{proof}
By the proof of Proposition \ref{prouno}, we know that if $q$ is even then the elements in ${\cal C}_1^{(q-1,q^2)} \setminus \{\alpha^{\frac{q}{2}}\}$ are the roots of all the monic irreducible polynomials of degree two in $\bbbf_q[x]$ that divide $x^{q+1}-\alpha$. Therefore $NI_q(x^{q+1}-\alpha,2)=(|{\cal C}_1^{(q-1,q^2)}|-1)/2=\frac{q}{2}$. Similarly, if $q$ is odd, then since $|{\cal C}_1^{(q-1,q^2)} \cap \bbbf_q^*|=0$, we have $NI_q(x^{q+1}-\alpha,2)=|{\cal C}_1^{(q-1,q^2)}|/2=\frac{q+1}{2}$. 
\end{proof}

\section{Some results regarding the number of irreducible and primitive polynomials of the form $x^2+\alpha^ix+\alpha$}\label{seccuatro}

At this point, we want to determine how many distinct irreducible and primitive polynomials of the form $x^2+\alpha^ix+\alpha$ exist, where $\alpha$ is our fixed primitive element of $\bbbf_{q}$ and $0\leq i<q-1$. To do this, we first need a simple way to determine whether or not a polynomial of the form $x^2 +\alpha^i x+\alpha \in \bbbf_{q}[x]$ is irreducible (or primitive). Thus, following the approach in \cite{Vega}, we define the disjoint integer sets:

{\small
\begin{eqnarray}
R_0&:=&\{\:0\leq i<q-1\:|\:\alpha^{i}=\alpha^{k+1}+\alpha^{q-k-1}, \mbox{ with } 0\leq k<\frac{q}{2}-1\:\}\; \mbox{ and}\nonumber \\
B_{0}&:=&\left\{\:0,1,2,\ldots,q-2\:\right\} \setminus R_0\;, \nonumber
\end{eqnarray}
}

\noindent
if $q$ is even and if $q$ is odd, we define instead the sets:

{\small
\begin{eqnarray}
R_1&:=&\left\{\:\residuo{i}{\frac{q-1}{2}}\:|\:\alpha^{i}=\alpha^{k+1}+\alpha^{q-k-1}\;,\mbox{ with } 0\leq k<\left\lfloor \frac{q-1}{4} \right\rfloor\:\right\}\; \mbox{ and}\nonumber \\
B_{1}&:=&\left\{\:0,1,2,\ldots,\frac{q-1}{2}-1\:\right\} \setminus R_1\;, \nonumber
\end{eqnarray}
}

\noindent
where $\residuo{i}{\frac{q-1}{2}}$ is the remainder of the division of $i$ by $\frac{q-1}{2}$, and $\lfloor v \rfloor$ is the largest integer less than or equal to $v$. Recall that $\delta=0$ if $q$ is even, and $\delta=1$ otherwise. Given the definition of $B_{\delta}$, it is easy to see that $x^2+\alpha^i x+\alpha$ is irreducible iff either $i \in B_{0}$ if $\delta=0$, or $\residuo{i}{\frac{q-1}{2}} \in B_{1}$ if $\delta=1$.

Since any primitive polynomial is irreducible, to determine whether an irreducible polynomial of the form $x^2 +\alpha^i x+\alpha \in \bbbf_{q}[x]$ is primitive, we also require the following set:

$$I_{\delta} := \left\{ \:i \in B_{\delta} \:|\: x^2+\alpha^i x+\alpha \in \bbbf_{q}[x], \mbox{ is primitive}\:\right\} \subseteq B_{\delta} \;.$$

Suppose $I_{\delta}=B_{\delta}$. For this particular case, we ask: Is there a simple way to determine whether an irreducible polynomial of the form $x^2+bx+c \in \bbbf_{q}[x]$ is primitive? That is, simply by looking at the two coefficients $b$ and $c$, is there a simple way to determine whether an irreducible polynomial of the form $x^2+bx+c \in \bbbf_{q}[x]$ is primitive? The answer is yes, and the following result, which is similar to \cite[Theorem 6]{Vega}, shows this. 
 
\begin{proposition}\label{prodos}
Let $c$ be a primitive element in $\bbbf_q$ and let $0 \leq j<q-1$ be the smallest integer such that $c=\alpha^{2j+1}$. Let $i$ be an integer such that $0 \leq i<q-1$ if $q$ is even, and $0 \leq i<\frac{q-1}{2}$ otherwise. Then the two polynomials $f(x)=x^2 \pm \alpha^{i+j}x+\alpha^{2j+1} \in \bbbf_{q}[x]$ are irreducible iff $i\in B_{\delta}$ and primitive iff $i\in I_{\delta}$. Therefore any irreducible polynomial $x^2+bx+c \in \bbbf_{q}[x]$, with $b\neq 0$ and $c$ a primitive element of $\bbbf_q$, is primitive iff $B_{\delta}=I_{\delta}$.
\end{proposition}

\begin{proof}
By \cite[Proposition 1]{Vega}, we can, without loss of generality, assume that $f(x)=x^2+\alpha^{i+j}x+\alpha^{2j+1}$. Suppose that $f(x)$ is irreducible. We claim that the polynomial $x^2+\alpha^{i}x+\alpha$ is also irreducible. Assume for the sake of contradiction that it is not, and let $r_1,r_2 \in \bbbf_q^*$ be such that $(x-r_1)(x-r_2)=x^2+\alpha^{i}x+\alpha$, where $-r_1-r_2=\alpha^{i}$, and $r_1r_2=\alpha$. Hence, $(x-\alpha^{j}r_1)(x-\alpha^{j}r_2)=x^2+\alpha^{i+j}x+\alpha^{2j+1}$, a contradiction! Consequently, $i \in B_{\delta}$. 

Conversely, if $i \in B_{\delta}$, then $x^2+\alpha^{i}x+\alpha$ is irreducible. Assume that $x^2+\alpha^{i+j}x+\alpha^{2j+1}$ is not irreducible. Let $r_1,r_2 \in \bbbf_q^*$ such that $(x-r_1)(x-r_2)=x^2+\alpha^{i+j}x+\alpha^{2j+1}$. Thus $-r_1-r_2=\alpha^{i+j}$, $r_1r_2=\alpha^{2j+1}$ and, consequently, $(x-r_1/\alpha^{j})(x-r_2/\alpha^{j})=x^2+\alpha^{i}x+\alpha$, a contradiction! 

Suppose that $x^2+\alpha^{i+j}x+\alpha^{2j+1} \in \bbbf_{q}[x]$ is primitive. Then by \cite[Theorem 3.18]{Lidl}, $\qord(x^2+\alpha^{i+j}x+\alpha^{2j+1})=q+1$ and by \cite[Corollary 1]{Vega} (take therein $b=\alpha^i$ and $c=m=\alpha$), we have 

$$q+1=\qord(x^2+\alpha^{i+j}x+\alpha^{2j+1})=\qord(x^2+\alpha^{i}x+\alpha)\;.$$

Since $\alpha$ is a primitive element of $\bbbf_{q}$, it follows again from \cite[Theorem 3.18]{Lidl}, that the polynomial $x^2+\alpha^{i}x+\alpha$ is primitive.  In consequence, $i \in I_{\delta}$.

Conversely, if $i \in I_{\delta}$, then $x^2+\alpha^{i}x+\alpha$ is primitive. Thus, by \cite[Theorem 3.18]{Lidl}, $\qord(x^2+\alpha^{i}x+\alpha)=q+1$ and by \cite[Corollary 1]{Vega} (again take therein $b=\alpha^i$ and $c=m=\alpha$), we have 

$$\qord(x^2+\alpha^{i+j}x+\alpha^{2j+1})=\qord(x^2+\alpha^{i}x+\alpha)=q+1\;.$$

Since $\alpha^{2j+1}=c$ is a primitive element of $\bbbf_{q}$, it follows again from \cite[Theorem 3.18]{Lidl} that the polynomial $x^2+\alpha^{i+j}x+\alpha^{2j+1}$ is primitive.  

Suppose that any irreducible polynomial $x^2+bx+c \in \bbbf_{q}[x]$, with $b\neq 0$ and $c$ a primitive element of $\bbbf_q$, is primitive. Let $i \in B_{\delta}$. Since $\alpha^i \neq 0$ and $\alpha$ is a primitive element of $\bbbf_q$, $x^2+\alpha^i x+\alpha$ is primitive. Thus, $i \in I_{\delta}$, $B_{\delta} \subseteq I_{\delta}$, and $I_{\delta} = B_{\delta}$.

Conversely, suppose that $x^2+bx+c=x^2+bx+\alpha^{2j+1} \in \bbbf_{q}[x]$, is irreducible. Since $b\neq 0$, let $i$ be the smallest integer such that $\alpha^i= \pm b/\alpha^j$. Thus, the two polynomials $x^2 \pm \alpha^{i+j}x+\alpha^{2j+1}=x^2 \pm bx+c$ are irreducible, and $i \in B_{\delta}$. But $B_{\delta}=I_{\delta}$, therefore $i \in I_{\delta}$ and $x^2+bx+c$ is primitive.
\end{proof}

At this point, by relating the previous result to Theorem \ref{teoceroA}, we can now conclude that if the sufficient conditions in such a theorem are satisfied, then $I_{\delta}\;=\;B_{\delta}$. Clearly, our goal now is to prove that the converse is also true. Before doing so, we need to determine the cardinalities of the sets $B_{\delta}$ and $I_{\delta}$:

\begin{proposition}\label{protres}
Assume our current notation. Then

\begin{eqnarray}
|B_{\delta}|&=&\left\{ \begin{array}{cl}
		\;\frac{q}{2} & \mbox{ if $\delta=0$}, \\
\\
		\;\left\lfloor \frac{q+1}{4} \right\rfloor & \mbox{ if $\delta=1$},
			\end{array}
\right . \nonumber \\
|I_{\delta}|&=&\phi(q+1)/2\;, \nonumber
\end{eqnarray}

\noindent
where $\phi$ denotes Euler’s totient function.
\end{proposition}

\begin{proof}
Note that $|B_{\delta}|$ is the number of irreducible polynomials of the form $x^2+\alpha^i x+\alpha$, where $0\leq i<q-1$ if $\delta=0$, and $0\leq i<\frac{q-1}{2}$ if $\delta=1$. According to Proposition \ref{prouno}, every irreducible polynomial of degree two dividing $x^{q+1}-\alpha$ has the form $x^2+bx+\alpha$ for some $b\in \bbbf_q$. Furthermore, observe that $x^2+\alpha$ fails to be irreducible when $q$ is even. Thus, by Corollary \ref{coruno}, we have $|B_{0}|=NI_q(x^{q+1}-\alpha,2)=\frac{q}{2}$ if $q$ is even. On the other hand, if $q$ is odd, we know, by Proposition \ref{prouno}, that the polynomial $x^2+\alpha$ is an irreducible divisor of $x^{q+1}-\alpha$ iff $\frac{q+1}{2}$ is odd. In addition, note that for any $i\in B_{1}$, the two polynomials $x^2+\alpha^i x+\alpha$ and $x^2-\alpha^i x+\alpha=x^2+\alpha^{\frac{q-1}{2}+i} x+\alpha$ are irreducible. Therefore, if $q$ is odd, we have $|B_{1}|=\lfloor NI_q(x^{q+1}-\alpha,2)/2 \rfloor =\lfloor \frac{q+1}{4} \rfloor$. 

By \cite[Theorem 6]{Vega}, the total number of primitive polynomials over $\bbbf_q$ of degree two is $|I_{\delta}| |J_{\delta}| 2^{\delta}$, where $|J_{\delta}|$ is the number of primitive elements in $\bbbf_q$. On the other hand, since any primitive element in $\bbbf_{q^2}$ is a root of a primitive polynomial of degree two over $\bbbf_q$, and given that there are $\phi(q^2-1)$ primitive elements in $\bbbf_{q^2}$, the total number of such polynomials is $\phi(q^2-1)/2$. Therefore, the total number of primitive polynomials over $\bbbf_q$ of degree two is equal to $\phi(q^2-1)/2=\phi(q+1)\phi(q-1)\gcd(q+1,q-1)/2$. Thus, the result now follows from the fact that $|J_{\delta}|=\phi(q-1)$ and $2^{\delta}=\gcd(q+1,q-1)$. 
\end{proof}

\begin{remark}
Note that if $q$ is odd ($\delta=1$), then $\left\lfloor \frac{q+1}{4} \right\rfloor=\frac{q-1}{2}-\left\lfloor \frac{q-1}{4} \right\rfloor$. Thus, although a formal proof has not been presented, the result concerning $|B_{1}|$ is given in \cite[Section 5]{Vega}.
\end{remark}

\section{Towards our main result}\label{seccinco}

We have already concluded that if the sufficient conditions in Theorem \ref{teoceroA} hold, then we have $I_{\delta}\;=\;B_{\delta}$ (see Section \ref{seccuatro}). The following result shows that the converse is also true.

\begin{theorem}\label{teouno}
Let $B_{\delta}$ and $I_{\delta}$ be as before. Then $B_{\delta}=I_{\delta}$ iff $q+1$ is either of the form $q+1=\pi$, $q+1=2 \pi$, or $q+1=2^t$, for some prime $\pi>2$ and some positive integer $t$.
\end{theorem}

\begin{proof}
Since $B_{\delta} \supseteq I_{\delta}$, $|B_{\delta}| \geq |I_{\delta}|$. Suppose that $q+1=d$, where $d$ is an odd integer. Let $\phi$ be as in Proposition \ref{protres}. Thus $q$ is even and, by Proposition \ref{protres}, $|B_{\delta}|=\frac{q}{2}=\frac{d-1}{2} \geq |I_{\delta}|=\phi(q+1)/2=\phi(d)/2$, and clearly $|B_{\delta}| = |I_{\delta}|$ iff $\frac{d-1}{2}=\phi(d)/2$ iff $d=\pi$ for some prime $\pi>2$. Now, assume that $q+1=2d$, with $d$ is an odd integer. Thus $q$ is odd and, by Proposition \ref{protres}, $|B_{\delta}|=\lfloor \frac{q+1}{4} \rfloor=\frac{q-1}{4}=\frac{d-1}{2} \geq |I_{\delta}|=\phi(q+1)/2=\phi(d)/2$, and again $|B_{\delta}| = |I_{\delta}|$ iff $\frac{d-1}{2}=\phi(d)/2$ iff $d=\pi$ for some prime $\pi>2$. Finally, for some positive integer $t>1$, suppose that $q+1=2^t d$, with $d$ is an odd integer. Thus $q$ is odd and, by Proposition \ref{protres}, $|B_{\delta}|=\lfloor \frac{q+1}{4} \rfloor=\frac{q+1}{4}=2^{t-2} d \geq |I_{\delta}|=\phi(q+1)/2=2^{t-2}\phi(d)$, and clearly $|B_{\delta}| = |I_{\delta}|$ iff $2^{t-2} d=2^{t-2}\phi(d)$ iff $d=1$. Since $B_{\delta} \supseteq I_{\delta}$, the result now follows from the fact that $|B_{\delta}|=|I_{\delta}|$ iff $B_{\delta}=I_{\delta}$. 
\end{proof}

In Theorem \ref{teoceroA} it is shown that any irreducible polynomial $f(x)=x^2+bx+c \in \bbbf_{q}[x]$, with $b\neq 0$ and $c$ a primitive element of $\bbbf_q$, is a primitive polynomial over $\bbbf_q$ if $q+1$ is either of the form $q+1=2^t$, $q+1=\pi$, or $q+1=2 \pi$, for some positive integer $t$ and some prime $\pi>2$. Now it is time to prove that the sufficient conditions in Theorem \ref{teoceroA} are also necessary. 

\begin{theorem}\label{MiCar1}
Let $q$ be a power of a prime number. Then, $q+1$ is either of the form $q+1=2^t$, $q+1=\pi$, or $q+1=2 \pi$, for some positive integer $t$ and some prime $\pi>2$ iff any irreducible polynomial of the form $f(x)=x^2+bx+c \in \bbbf_{q}[x]$, with $b\neq 0$ and $c$ a primitive element of $\bbbf_q$, is primitive.
\end{theorem}

\begin{proof}
By Theorem \ref{teoceroA}, we know that if $q+1$ is either of the form $q+1=2^t$, $q+1=\pi$, or $q+1=2 \pi$, for some positive integer $t$ and some prime $\pi>2$, then any irreducible polynomial of the form $f(x)=x^2+bx+c \in \bbbf_{q}[x]$, with $b\neq 0$ and $c$ a primitive element of $\bbbf_q$, is a primitive polynomial. On the other hand, by Proposition \ref{prodos} any irreducible polynomial $x^2+bx+c \in \bbbf_{q}[x]$, with $b\neq 0$ and $c$ a primitive element of $\bbbf_q$, is primitive iff $B_{\delta}=I_{\delta}$. Finally, by Theorem \ref{teouno}, $B_{\delta}=I_{\delta}$ iff $q+1$ is either of the form $q+1=2^t$, $q+1=\pi$, or $q+1=2 \pi$, for some positive integer $t$ and some prime $\pi>2$.
\end{proof}

\section{A new family of finite fields for which it is easy to determine when an irreducible polynomial is primitive}\label{secseis}

Suppose that $q$, the order of the finite field $\bbbf_q$, satisfies $q+1=4 \pi$, for some prime $\pi>2$. Then, in this case and using the following result, it is also easy to determine when an irreducible polynomial of the form $f(x)=x^2+bx+c \in \bbbf_{q}[x]$, where $c$ is a primitive element of $\bbbf_q$, is primitive.

\begin{theorem}\label{teodiez}
Let $q$ be a power of a prime number. If $q+1=4 \pi$, for some prime $\pi>2$, then any irreducible polynomial of the form $f(x)=x^2+bx+c \in \bbbf_{q}[x]$, where $c$ a primitive element of $\bbbf_q$, is a primitive polynomial iff $b^2 \neq 2c$. In addition, if $b^2=2c$ then $f(x)$ is an irreducible non-primitive polynomial such that \qord$(f(x))=4$ and \ord$(f(x))=4(q-1)$.
\end{theorem}

\begin{proof}
Clearly $q$ is odd. Since $q+1=4 \pi$ and $c$ is a primitive element of $\bbbf_q$, $-c$ is a square of an element of $\bbbf_q^*$. Hence $x^2+c$ is reducible over $\bbbf_q$ and $b \neq 0$. 

Let $h(x), r(x) \in \bbbf_q [x]$ be the uniquely polynomials such that

$$h(x)=\frac{x^{5}-r(x)}{f(x)}=h_{1}x^{3}+h_{2}x^{2}+h_{3}x+h_{4}\;,$$

\noindent
and $\deg(r(x))<2$. Suppose that $f(x)$ is a primitive polynomial. Then by Condition {\rm (B)} of Theorem \ref{MiCar2}, $p=\pi$, $\frac{q+1}{p}=4$, and the four terms $h_i$, $i=1,2,3,4$, are non-zero. Following the argument in the proof of Proposition \ref{prodosA}, it follows that $h_1=1$, $h_2=-b$, $h_3=-(h_2b+h_1c)=b^2-c$ and $h_4=-(h_3b+h_2c)=-b(b^2-2c) \neq 0$. Therefore $b^2 \neq 2c$. Conversely, suppose now that $b^2 \neq 2c$. Since $h_2=-b\neq 0$ and $c$ is a primitive element of $\bbbf_q$, we have $h_3=b^2-c\neq 0$. Thus the four terms $h_i$, $i=1,2,3,4$, are non-zero. According to Condition {\rm (B)} of Theorem \ref{MiCar2}, $f(x)$ is therefore a primitive polynomial.

Let $a \in \bbbf_q^*$ and $\rho, r \in \bbbz^+$ be the uniquely determined values such that $x^{\rho} \equiv a \pmod{f(x)}$, $\rho=\qord(f(x))$, and $r=\ord(f(x))$. By \cite[Proposition 1]{Vega}, $f(x) | (x^{q+1}-c)$ and therefore, owing to Remark \ref{remuno}, we have $\rho|(q+1)$ and $c=a^{(q+1)/\rho}$. 

Let $g(x)$ be as in Proposition \ref{prodosA}. If $b^2=2c$ then, again by that proposition and its proof, $g_1=1$, $g_2=-b\neq 0$, $g_3=b^2-c\neq 0$, $g_4=-b(b^2-2c)=0$ and $\rho=4$. On the other hand, owing to Theorem \ref{teocero}, $r=\gcd(r,q-1)\rho$. Suppose, for the sake of contradiction, that $r=\ord(f(x))=\gcd(r,q-1)4 \neq 4(q-1)$. Thus $\gcd(r,q-1)<q-1$ and, by Proposition \ref{procero}, $a$ is not a primitive element in $\bbbf_q$. But this is a contradiction because $c$ is a primitive element of $\bbbf_q$ and $c=a^{(q+1)/\rho}=a^{\pi}$. 
\end{proof}


\begin{center}
Table I \\
{\em The first five finite fields $\bbbf_q$ that satisfy the condition in Theorem \ref{teodiez}.}
\end{center}
\begin{center}
\begin{tabular}{|ccccc|} \hline
{\bf $\bbbf_{11}$} & {\bf $\bbbf_{19}$} & {\bf $\bbbf_{27}$} & {\bf $\bbbf_{43}$} & {\bf $\bbbf_{67}$} \\ \hline
\end{tabular}
\end{center}
 
The first five finite fields $\bbbf_q$ that satisfy the condition in Theorem \ref{teodiez} are shown in Table I.

\begin{example}
Let $(q,c)=(11,2)$ and note that $q+1=12=4(3)$ and $\langle 2 \rangle=\bbbf_{11}^*$. It is not difficult to see that the polynomial $x^{12}-2 \in \bbbf_{11}[x]$ is factored as the product of six irreducible polynomials of degree two. In fact,

$$x^{12}-2=(x^2 \pm 2x+2)(x^2 \pm 4x+2)(x^2 \pm 5x+2)\;.$$

\noindent
Thus, by Theorem \ref{teoonce}, the four polynomials $(x^2 \pm 4x+2)$ and $(x^2 \pm 5x+2)$ are primitive, while the two polynomials $(x^2 \pm 2x+2)$ are non-primitive. For the latter two, their quasi-order and order are $4$ and $40$, respectively.
\end{example}

If $q+1=4 \pi$ for some prime $\pi>2$, then, using the following result, we can identify and obtain all non-primitive irreducible polynomials $f(x)=x^2+bx+c$, with $b\neq 0$ and $c$ a primitive element of $\bbbf_q$.

\begin{theorem}\label{teoonce}
Suppose that $q+1=4 \pi$ for some prime $\pi>2$. Then, for any primitive element $c$ in $\bbbf_q$, $2c$ is a square of an element of $\bbbf_q^*$ and the two polynomials $f(x)=x^2 \pm (2c)^{1/2}x+c$ are irreducible and non-primitive, where \qord$(f(x))=4$ and \ord$(f(x))=4(q-1)$.
\end{theorem}

\begin{proof}
Since $q$ is odd, $\delta=1$. Let $c$ be a primitive element in $\bbbf_q$ and let $0 \leq j<q-1$ be the smallest integer such that $c=\alpha^{2j+1}$. Let $i$ be an integer such that $0 \leq i<\frac{q-1}{2}$. Then, by Proposition \ref{prodos}, the two polynomials $x^2 \pm \alpha^{i+j}x+c \in \bbbf_q[x]$ are irreducible iff $i\in B_{1}$ and primitive iff $i\in I_{1}$.

On the other hand, by Proposition \ref{protres}, $|B_{1}|=\lfloor \frac{q+1}{4} \rfloor=\frac{q+1}{4}=\pi$ and $|I_{1}|=\phi(q+1)/2=\phi(\pi)=\pi-1$. Thus, it is clear that $|B_{1}|=|I_{1}|+1$. Since $B_{1} \supseteq I_{1}$, this in turn means that there exists a unique $i' \in B_{1}$ and $i' \notin I_{1}$ such that the two polynomials $f(x)=x^2 \pm \alpha^{i'+j}x+c \in \bbbf_q[x]$ are irreducible and non-primitive. However, by Theorem \ref{teodiez}, any irreducible polynomial of the form $x^2+bx+c$ is non-primitive iff $b^2=2c$. Therefore $(\pm \alpha^{i'+j})^2=2c$. The last part follows directly from Theorem \ref{teodiez}.
\end{proof}

\begin{example}
Let $q=11$ again and note that $\langle 2 \rangle=\langle 6 \rangle=\langle 7 \rangle=\langle 8 \rangle=\bbbf_{11}^*$. Thus, as asserted in Theorem \ref{teoonce}, we clearly have $2=(2(2))^{1/2}$, $1=(2(6))^{1/2}$, $5=(2(7))^{1/2}$, and $4=(2(8))^{1/2}$. Therefore the eight polynomials $(x^2 \pm 2x+2)$, $(x^2 \pm x+6)$, $(x^2 \pm 5x+7)$, and $(x^2 \pm 4x+8)$ are irreducible non-primitive polynomials, whose quasi-order and order are $4$ and $40$, respectively. In fact, since $\phi(q-1)=\phi(10)=4$, note that beyond these eight, there is no irreducible polynomial of the form $x^2+bx+c$, where $c$ is a primitive element of $\mathbb{F}_q$, with quasi-order $4$ and order $40$.
\end{example}

\section{Conclusion}\label{secsiete}

It was shown in \cite{Vega} that there exist at least three infinite families of finite fields for which, given an irreducible polynomial of the form $f(x)=x^2+bx+c$, it is very easy to determine whether it is primitive simply by looking at its two coefficients $b$ and $c$. The main objective of this work was to show that, beyond these three infinite families of finite fields, there is no other finite field $\bbbf_q$ such that, given any irreducible polynomial of the form $f(x)=x^2+bx+c$, it will be primitive if and only if $b\neq 0$ and $c$ is a primitive element of that finite field. That is, the main objective of this work was to formally prove in Theorem \ref{MiCar1} that the sufficient conditions in Theorem \ref{teoceroA} are also necessary.

Let $b,c \in \bbbf_q$, such that $b\neq 0$ and $c$ is a primitive element of $\bbbf_q$. As another result, which can be considered as an extension of Theorem \ref{MiCar2}, we determine the order and the quasi-order of any irreducible polynomial of the form $f(x)=x^2+bx+c \in \bbbf_{q}[x]$. That is, if the sufficient conditions in Theorem \ref{teoceroA} do not hold, Theorem \ref{MiCar2E} shows that the quasi-order and the order of any irreducible polynomial of the form $f(x)=x^2+bx+c \in \mathbb{F}_{q}[x]$, with $b\neq 0$ and $c$ a primitive element of $\mathbb{F}_q$, can be determined simply by examining the $\frac{q+1}{p}$ terms of the polynomial $h(x)$.

As a final result, another infinite family of finite fields $\bbbf_q$ is presented for which, using a different test on the two coefficients $b$ and $c$, it is also easy to determine when a monic irreducible polynomial of degree two is primitive. More specifically, Theorem \ref{teodiez} proves that for any finite field $\bbbf_q$ satisfying $q+1=4 \pi$, for some prime $\pi>2$, every irreducible polynomial of the form $f(x)=x^2+bx+c \in \bbbf_{q}[x]$, where $c$ is a primitive element of $\bbbf_q$, is primitive if and only if $b^2 \neq 2c$. Finally, for finite fields of this kind and any $b,c \in \bbbf_q$ with $b\neq 0$ and $c$ a primitive element of $\bbbf_q$, we identify, in Theorem \ref{teoonce}, all the irreducible non-primitive polynomials of the form $f(x)=x^2+bx+c$ and determine their quasi-order and order.

As future work, it would be interesting to determine whether the sufficient condition in Theorem \ref{teodiez} is also necessary. In addition, it is worth exploring other infinite families of finite fields $\mathbb{F}_q$ for which it is possible to determine when a monic irreducible polynomial of degree two is primitive, using different tests involving only the two coefficients $b$ and $c$. Finally, we believe that using the terms of a polynomial $h(x)$, similar to the one in Theorem \ref{MiCar2E}, one can determine the quasi-order and order of any irreducible polynomial of degree two over any finite field. It would be interesting to prove this.


\end{document}